\documentclass[12pt]{amsart}

\usepackage{amsmath}
\usepackage{amsthm}
\usepackage{amssymb}
\usepackage{caption}
\usepackage{xspace}
\usepackage{latexsym}
\usepackage[curve]{xypic} 
\usepackage{cite}
\usepackage{enumitem}
\usepackage{mathtools}
\usepackage{color}
\usepackage{tikz}
\usepackage{tikz-cd}
\usepackage{hyperref} 
\usepackage{verbatim}
\usepackage{xcolor}
\usepackage{caption}
\usepackage{subcaption}

\usepackage[all]{xy}

\numberwithin{equation}{section}

\setlist{itemsep=4pt, topsep=4pt}

\makeatletter

\def\chaptermark#1{}

\def\chapter{%
  \if@openright\cleardoublepage\else\clearpage\fi
  \thispagestyle{plain}\global\@topnum\z@
  \@afterindenttrue \secdef\@chapter\@schapter}

\def\@chapter[#1]#2{\refstepcounter{chapter}%
  \ifnum\c@secnumdepth<\z@ \let\@secnumber\@empty
  \else \let\@secnumber\thechapter \fi
  \typeout{\chaptername\space\@secnumber}%
  \def\@toclevel{0}%
  \ifx\chaptername\appendixname \@tocwriteb\tocappendix{chapter}{#2}%
  \else \@tocwriteb\tocchapter{chapter}{#2}\fi
  \chaptermark{#1}%
  \addtocontents{lof}{\protect\addvspace{10\p@}}%
  \addtocontents{lot}{\protect\addvspace{10\p@}}%
  \@makechapterhead{#2}\@afterheading}
\def\@schapter#1{\typeout{#1}%
  \let\@secnumber\@empty
  \def\@toclevel{0}%
  \ifx\chaptername\appendixname \@tocwriteb\tocappendix{chapter}{#1}%
  \else \@tocwriteb\tocchapter{chapter}{#1}\fi
  \chaptermark{#1}%
  \addtocontents{lof}{\protect\addvspace{10\p@}}%
  \addtocontents{lot}{\protect\addvspace{10\p@}}%
  \@makeschapterhead{#1}\@afterheading}
\newcommand\chaptername{Chapter}

\def\@makechapterhead#1{\global\topskip 7.5pc\relax
  \begingroup
  \fontsize{\@xivpt}{18}\bfseries\centering
    \ifnum\c@secnumdepth>\m@ne
      \leavevmode \hskip-\leftskip
      \rlap{\vbox to\z@{\vss
          \centerline{\normalsize\mdseries
              \uppercase\@xp{\chaptername}\enspace\thechapter}
          \vskip 3pc}}\hskip\leftskip\fi
     #1\par \endgroup
  \skip@34\p@ \advance\skip@-\normalbaselineskip
  \vskip\skip@ }
\def\@makeschapterhead#1{\global\topskip 7.5pc\relax
  \begingroup
  \fontsize{\@xivpt}{18}\bfseries\centering
  #1\par \endgroup
  \skip@34\p@ \advance\skip@-\normalbaselineskip
  \vskip\skip@ }
\def\appendix{\par
  \c@chapter\z@ \c@section\z@
  \let\chaptername\appendixname
  \def\thechapter{\@Alph\c@chapter}}

\newcounter{chapter}

\newif\if@openright

\makeatother

\makeatletter 
\def\@cite#1#2{{\m@th\upshape\bfseries%
[{#1\if@tempswa{\m@th\upshape\mdseries, #2}\fi}]}}
\makeatother 

\theoremstyle{plain}
\newtheorem{thm}{Theorem}[section]
\newtheorem{cor}[thm]{Corollary}

\newtheorem{prop}[thm]{Proposition}

\newtheorem{theorem}[thm]{Theorem}
\theoremstyle{definition}

\newtheorem{definition}[thm]{Definition}
\theoremstyle{remark}
\newtheorem{rem}[thm]{Remark}

\newtheorem*{remark}{Remark}

\numberwithin{equation}{subsection}
\newcommand{\nc}{\newcommand}
\newcommand{\rnc}{\renewcommand}

\newcommand{\Teich}[1][g,n]{\mathcal{T}_{#1}}
\newcommand{\TeichV}{\mathcal{T}_{g,n}\times V_{\Sigma, S}}
\newcommand{\Mod}[1][g,n]{\mathrm{Mod}_{#1}}
\newcommand{\moduli}[1][g,n]{\mathcal{M}_{#1}}

\newcommand{\TeichD}{\widetilde{\mathcal{D}}_{g,n}}
\newcommand{\ModD}{{\mathcal{D}}_{g,n}}

\newcommand{\TeichA}{\widetilde{\mathcal{A}}_{g,n}}
\newcommand{\ModA}{{\mathcal{A}}_{g,n}}
\newcommand{\ModF}{{\mathcal{F}}_{g,n}}

\nc\bA{\mathbb{A}}
\nc\bB{\mathbb{B}}
\nc\bC{\mathbb{C}}
\nc\bD{\mathbb{D}}
\nc\bE{\mathbb{E}}
\nc\bF{\mathbb{F}}
\nc\bG{\mathbb{G}}
\nc\bH{\mathbb{H}}
\nc\bI{\mathbb{I}}
\nc{\bJ}{\mathbb{J}} 
\nc\bK{\mathbb{K}}
\nc\bL{\mathbb{L}}
\nc\bM{\mathbb{M}}
\nc\bN{\mathbb{N}}
\nc\bO{\mathbb{O}}
\nc\bP{\mathbb{P}}
\nc\bQ{\mathbb{Q}}
\nc\bR{\mathbb{R}}
\nc\bS{\mathbb{S}}
\nc\bT{\mathbb{T}}
\nc\bU{\mathbb{U}}
\nc\bV{\mathbb{V}}
\nc\bW{\mathbb{W}}
\nc\bY{\mathbb{Y}}
\nc\bX{\mathbb{X}}
\nc\bZ{\mathbb{Z}}
\nc\cA{\mathcal{A}}
\nc\cB{\mathcal{B}}
\nc\cC{\mathcal{C}}
\rnc\cD{\mathcal{D}}
\nc\cE{\mathcal{E}}
\nc\cF{\mathcal{F}}
\nc\cG{\mathcal{G}}
\rnc\cH{\mathcal{H}}
\nc\cI{\mathcal{I}}
\nc{\cJ}{\mathcal{J}} 
\nc\cK{\mathcal{K}}
\rnc\cL{\mathcal{L}}
\nc\cM{\mathcal{M}}
\nc\cN{\mathcal{N}}
\nc\cO{\mathcal{O}}
\nc\cP{\mathcal{P}}
\nc\cQ{\mathcal{Q}}
\rnc\cR{\mathcal{R}}
\nc\cS{\mathcal{S}}
\nc\cT{\mathcal{T}}
\nc\cU{\mathcal{U}}
\nc\cV{\mathcal{V}}
\nc\cW{\mathcal{W}}
\nc\cY{\mathcal{Y}}
\nc\cX{\mathcal{X}}
\nc\cZ{\mathcal{Z}}
\nc\bfA{\mathbf{A}}
\nc\bfB{\mathbf{B}}
\nc\bfC{\mathbf{C}}
\nc\bfD{\mathbf{D}}
\nc\bfE{\mathbf{E}}
\nc\bfF{\mathbf{F}}
\nc\bfG{\mathbf{G}}
\nc\bfH{\mathbf{H}}
\nc\bfI{\mathbf{I}}
\nc{\bfJ}{\mathbf{J}} 
\nc\bfK{\mathbf{K}}
\nc\bfL{\mathbf{L}}
\nc\bfM{\mathbf{M}}
\nc\bfN{\mathbf{N}}
\nc\bfO{\mathbf{O}}
\nc\bfP{\mathbf{P}}
\nc\bfQ{\mathbf{Q}}
\nc\bfR{\mathbf{R}}
\nc\bfr{\boldsymbol{r}}
\nc\bfS{\mathbf{S}}
\nc\bfT{\mathbf{T}}
\nc\bfU{\mathbf{U}}
\nc\bfV{\mathbf{V}}
\nc\bfW{\mathbf{W}}
\nc\bfY{\mathbf{Y}}
\nc\bfX{\mathbf{X}}
\nc\bfZ{\mathbf{Z}}

\nc\bfmu{\boldsymbol{\mu}}
\nc{\ModX}{\mathrm{Mod}(X)}
\nc{\QuadX}{\mathrm{Quad}(X)}
\nc{\Ztwo}{\mathbb{Z}/2\mathbb{Z}}
\nc{\del}{\partial}

\nc{\dmo}{\DeclareMathOperator}
\nc{\wt}{\widetilde}
\rnc{\Re}{\operatorname{Re}}
\rnc{\Im}{\operatorname{Im}}
\rnc{\span}{\operatorname{span}}
\dmo{\rank}{rank}
\dmo{\End}{End}
\dmo{\Hom}{Hom}
\dmo{\Jac}{Jac}
\dmo{\Id}{Id}
\dmo{\Ann}{Ann}
\dmo{\Area}{Area}
\dmo{\CP}{\bC P^1}
\dmo{\rk}{rk}
\dmo{\rel}{rel}
\dmo{\ra}{\rightarrow}
\dmo{\Twist}{\mathrm{Twist}}
\dmo{\TwistX}{\mathrm{Twist}(X, \omega)}
\dmo{\Pic}{Pic}
\dmo{\Res}{Res}

\rnc{\Col}{\operatorname{Col}}

\nc{\ColOne}{\Col_{\bfC_1}}
\nc{\ColOneX}{\ColOne(X,\omega)}

\nc{\ColTwo}{\Col_{\bfC_2}}
\nc{\ColTwoX}{\ColTwo(X,\omega)}

\nc{\ColThree}{\Col_{\bfC_3}}
\nc{\ColThreeX}{\ColThree(X,\omega)}

\nc{\ColOneTwo}{\Col_{\bfC_1, \bfC_2}}
\nc{\ColOneTwoX}{\ColOneTwo(X,\omega)}

\nc{\ColOneThree}{\Col_{\bfC_1, \bfC_3}}
\nc{\ColOneThreeX}{\ColOneThree(X,\omega)}

\nc{\MOne}{\cM_{\bfC_1}}
\nc{\MTwo}{\cM_{\bfC_2}}
\nc{\MOneTwo}{\cM_{\bfC_1, \bfC_2}}

\nc{\MThree}{\cM_{\bfC_3}}

\nc{\MOneThree}{\cM_{\bfC_1, \bfC_3}}

\dmo{\For}{\cF}

\nc{\GL}{\mathrm{GL}^+(2, \bR)}

\newcommand{\cx}{{\mathbb C}}

\newcommand{\zed}{\mathbb{Z}}
\newcommand{\bfm}{\boldsymbol{m}}

\renewcommand{\flat}[1][\chi, \bfm]{L_{#1}}
\DeclareMathOperator{\ord}{ord}

\newcommand{\Teichmuller}{Teichm\"uller\xspace}

\title[Affine and Dilation Surfaces]{Moduli Spaces of Complex Affine and Dilation Surfaces}
\author[Apisa]{Paul~Apisa}
\address{Department of Mathematics, University of Michigan, Ann Arbor, MI 48104, USA}
\email{paul.apisa@gmail.com}
\thanks{During the preparation of this paper, the first author was partially supported by NSF Postdoctoral Fellowship DMS 1803625.}

\author[Bainbridge]{Matt~Bainbridge}
\address{Department of Mathematics, Indiana University, Bloomington, IN 47405, USA}
\email{bainbridge.matt@gmail.com}
\thanks{Research of the second author is supported in part by the Simons Foundation grant \#713192.}

\author[Wang]{Jane~Wang}
\address{Department of Mathematics, Indiana University, Bloomington, IN 47405, USA}
\email{wangjan@iu.edu}

\date{\today}

\begin{document}
\maketitle
\thispagestyle{empty}

\begin{abstract}
We construct moduli spaces of complex affine and dilation surfaces. Using ideas of Veech \cite{veech93}, we show that the the moduli space $\ModA(\bfm)$ of genus $g$ affine surfaces with cone points of complex order $\bfm = (m_1\ldots, m_n)$ is a holomorphic affine bundle over $\moduli$, and the moduli space $\ModD(\bfm)$ of dilation surfaces is a covering space of $\moduli$.

We then classify the connected components of $\ModD(\bfm)$ and show that it is an orbifold-$K(G, 1)$, where $G$ is the framed mapping class group of \cite{CalderonSalter2020}.
\end{abstract}

\setcounter{tocdepth}{2} 
\tableofcontents

\vspace{-1cm}

\section{Introduction}

Given a Riemann surface $X$, a \emph{complex affine structure} is a
maximal atlas of charts whose transition functions belong to the group
$\mathrm{Aff}(\mathbb{C})$ of complex-affine automorphisms
$f(z) = az+b$ of the complex plane.  In other words, an affine
structure is an $(\mathrm{Aff}(\mathbb{C}), \mathbb{C})$-structure
compatible with the complex structure on $X$.  If $X$ is closed, such
an affine structure exists if and only if $X$ has genus one.

To obtain affine structures on higher genus closed surfaces, one must
allow isolated singularities.  A \emph{branched complex affine
  structure} on $X$ is an affine structure on the complement of a
discrete set $P$.  We require that this structure is meromorphic in
the sense that in local coordinates $z$ near each point
$p$ in $P$, the affine structure is of the form $z^m g(z)dz$ for
some complex number $m$ and nonzero holomorphic function $g$.  (The
charts of the affine structure near $p$ are given by branches of
$\int z^mg(z) dz$.)  We say that $p$ is a \emph{cone point of order
  $m$}.

If $X$ is closed of genus $g$, a branched complex affine structure
must have finitely many cone points whose orders sum to $2g-2$.  As we
are primarily interested in closed surfaces of genus $g>1$, where
every affine structure is branched, we will subsequently refer to
these simply as affine structures.

Let $\TeichA$ denote the Teichm\"uller space of affine surfaces, the
set of equivalence classes of affine surfaces $X$ having labeled cone
points $P = \{p_1, \ldots, p_n\}$, together with a homotopy class of
marking $(\Sigma_g, S) \to (X, P)$, where $\Sigma_g$ is a fixed genus
$g$ surface with a set $S$ of $n$ distinguished points.  Given a tuple
$\boldsymbol{m} = (m_i)_{i=1}^n$ of complex numbers that sum to
$2g-2$, let $\TeichA(\bfm)\subset \TeichA$ denote the locus of affine
surfaces where each $p_i$ has order $m_i$.  Similarly, define the
corresponding moduli spaces of affine surfaces $\ModA$ and
$\ModA(\bfm)$ as the respective quotients by the mapping class group
$\Mod$.  Since a complex affine surface has an underlying conformal
structure, there are natural forgetful maps $\TeichA\to \Teich$
and $\ModA\to \moduli$ to the corresponding Teichm\"uller and
moduli spaces, and similarly for $\TeichA(\bfm)$ and $\ModA(\bfm)$.

Veech constructed these moduli space of affine surfaces and proved
many foundational results on their structure in his monumental paper,
\cite{veech93}.  Given a closed Riemann surface $X$, Veech showed that
the set $A(X, P, \bfm)$ of affine structures on $X$, having fixed
cone points $P$ and orders $\bfm$ summing to $2g-2$, is itself an affine
space which is modeled on the space $\Omega(X)$ of holomorphic
one-forms on $X$.  This affine structure is induced by the
\emph{exponential action} of $\Omega(X)$ on $A(X, P, \bfm)$, which we
introduce in \S\,2.  Veech moreover showed that these moduli space of
affine structures are smooth complex manifolds.

Veech's results in other words tell us that the fibers of the
forgetful maps $\TeichA(\bfm)\to \Teich$ and $\ModA(\bfm)\to \moduli$
are $g$-dimensional bundles of affine spaces (mod the action of finite
automorphism groups in the second case).  This strongly suggests that
these moduli spaces are actually holomorphic affine bundles, though
Veech did not consider the global structure of these forgetful maps.
In \S\,3, we give a brief self-contained construction of these moduli
spaces as complex analytic manifolds, and show that they are in fact
affine bundles.

\begin{theorem}\label{thm:affine_covering}
  The forgetful maps $\TeichA(\bfm)\to \Teich$ and $\TeichA\to \Teich$
  are trivial holomorphic affine bundles, i.e. they have global
  holomorphic sections, with the fiber over a point representing a
  surface $(X, P)$ modeled on the space $\Omega(X)$ of holomorphic
  $1$-forms and, respectively, the space $\Omega(X, P)$ of holomorphic
  one-forms on $X$ which may have at worst simple poles along $P$.
  
  The forgetful maps $\ModA(\bfm)\to \moduli$ and $\ModA\to \moduli$
  are holomorphic affine bundles (in the orbifold category) whose
  fiber over a surface $(X, P)$ can be identified with
  $\Omega(X)/\mathrm{Aut}(X,P)$ and $\Omega(X, P)/\mathrm{Aut}(X,P)$
  respectively, where $\mathrm{Aut}(X,P)$ is the group of
  automorphisms of $X$ that fix $P$ pointwise. 
\end{theorem}

In other terms, we show that the sheaves of sections for these
forgetful maps are torsors for the sheaf of sections of the Hodge
bundle or the extended Hodge bundle, allowing simple poles at the
marked points, over the respective bases.

Let $\mathrm{Aff}_{\mathbb{R}_+}(\mathbb{C})$ denote the subgroup of
$\mathrm{Aff}(\mathbb{C})$ consisting of maps of the form
$f(z) = az + b$ where $a$ is a positive real number. An affine
surface is called a \emph{dilation surface} if the transition
functions for the atlas defining the
$(\mathrm{Aff}(\mathbb{C}), \mathbb{C})$-structure belong to
$\mathrm{Aff}_{\mathbb{R}_+}(\mathbb{C})$ (a choice of ``horizontal
direction'' is often taken to be part of the structure of a dilation
surface, but we do not do so here). We will let $\ModD(\bfm)$
denote the moduli space of all dilation surfaces on genus $g$ surfaces
with $n$ marked points whose complex cone angles are given by $\bfm$.

Though dilation surfaces do not appear in Veech's work \cite{veech93},
his results easily imply that the set $D(X, P, \bfm)$ of dilation
structures on $X$ with fixed cone points $P$ and orders $\bfm$ is an
affine set (or torsor) modeled on a lattice in $\Omega(X)$ (see
Corollary~\ref{C:DilationFiber}).  This strongly suggests that the
$\ModD(\bfm)$ are covering spaces of $\moduli$, which we also
establish in~\S\,3.

\begin{theorem}\label{thm:dilation_covering}
  Let $\bfm = (m_1, \ldots, m_n)$ be a tuple of complex numbers with
  integral real parts so that $\sum_{i=1}^n m_i = 2g-2$. Then
  $\ModD(\bfm)$ is a real analytic suborbifold of $\ModA(\bfm)$ whose forgetful map to  $\moduli$ is an orbifold covering map. 
\end{theorem}

By taking a directional vector field, a dilation surface structure on a compact genus $g$ Riemann surface
$X$ with $n$ cone points $P$ defines a \emph{framing} of
$X \setminus P$, that is the homotopy class of a trivialization
$$f\colon T^1(X \setminus P) \to (X \setminus P) \times S^1$$
of the tangent
circle bundle of $X\setminus P$. Given a $C^1$-immersed closed curve
$\gamma$, applying a framing $f$ to the tangent vector field of
$\gamma$ defines a ``Gauss map'' $G\colon S^1\to S^1$ whose degree is
the \emph{turning number} of $\gamma$ relative to $f$.  If $\gamma_i$
is a positively-oriented loop around the $i$th point in $P$ and $r_i$
is its turning number, then the Poincare-Hopf theorem implies that
$\sum_{i=1}^n r_i = 2g+n-2$.  Given a tuple of integers
$\boldsymbol{r} = (r_1, \hdots, r_n)$ that sums to $2g+n-2$, let
$F(X, P, \boldsymbol{r})$ denote the set of framings on
$X \setminus P$ so that the turning number of $\gamma_i$ is $r_i$.
For framings arising from dilation structures, the turning number and
order at a cone point are related by $r_i = \Re(m_i)+1$. 

Given integers $\bfr = (r_i)$ which sum to $2g - 2 +n$, we denote by
$\ModF(\bfr)$ the moduli space of $n$-pointed Riemann surfaces
$(X, P)$ together with a framing for $X\setminus P$ having turning
numbers $r_i$ at each $p_i$, a covering space $\ModF(\bfr)\to \moduli$ whose
fiber over the point represented by $X$ is a $H_1(X, \zed)$-torsor.
The map assigning to a dilation surface its framing defines a map
$\ModD(\bfm) \to \ModF(\bfr)$, where $r_i = \Re(m_i)+1$. 

\begin{theorem}\label{thm:dilation_covering2}
  The map $\ModD(\bfm) \to \ModF(\bfr)$ sending a dilation surface to
  its framing is an isomorphism of orbifold covering spaces over $\moduli$. 
\end{theorem}

Let $\Mod[g,n][f]< \Mod$ be the stabilizer of a framing $f$. 
Since $\moduli$ is an orbifold $K(\Mod, 1)$, we obtain:

\begin{cor}\label{C:FundamentalGroup}
  Each component $C$ of $\ModD(\bfm)$ is a
  $K(\mathrm{Mod}_{g,n}[f],1)$, in the orbifold sense,  where $f$ is the framing
  associated to an element of $C$.
\end{cor}

Calderon and Salter \cite{CalderonSalter2020} showed that the fundamental group of non-hyperelliptic components of strata of translation surfaces in genus $g\geq 5$ surjects onto $\mathrm{Mod}_{g,n}[f]$. It would be interesting to determine whether this follows from Corollary \ref{C:FundamentalGroup}.

By Theorem \ref{thm:dilation_covering2}, determining the components of $\ModD(\bfm)$ is equivalent to determining the $\mathrm{Mod}_{g,n}$ orbits\footnote{The related question of determining $\mathrm{Mod}_{g,n}$ orbits of relative framings was resolved by Randal-Williams \cite{RandalWilliams}; but this will not be relevant in the sequel.} on $F(X, P, \boldsymbol{r})$, which were computed by Kawazumi \cite{kawazumi} and which we will use to show the following.

\begin{theorem}\label{T:Components}
Let $\bfm = (m_1, \ldots, m_n)$ be as in Theorem \ref{thm:dilation_covering} and $\boldsymbol{\kappa} = (\mathrm{Re}(m_1), \hdots, \mathrm{Re}(m_n))$. 
\begin{itemize}
    \item[(a)] If $g=0$, $\ModD(\bfm)$ has exactly one component.
    \item[(b)] If $g = 1$ and $n = 0$, $\ModD(\bfm)$ has infinitely many components.
    \item[(c)] If $g = 1$ and $n \ne 0$, $\ModD(\bfm)$ has exactly $\varphi(\mathrm{gcd}(\boldsymbol{\kappa}))$ components, where $\varphi$ is Euler's phi function.
    \item[(d)] If $g \geq 2$ and $\boldsymbol{\kappa}$ contains an odd number, then $\ModD(\bfm)$ has exactly one component.
    \item[(e)] If $g \geq 2$ and $\boldsymbol{\kappa}$ contains only even numbers, then $\ModD(\bfm)$ has two components distinguished by the Arf invariant of the framings associated to the dilation surfaces in $\ModD(\bfm)$.
\end{itemize}
\end{theorem}

Theorem \ref{T:Components} answers a question of Duryev, Fougeron, and Ghazouani \cite{DFG} on the number of connected components of $\mathcal{D}_{g,n}(\bfm)$.

We conclude this section with a discussion of the previous literature. The study of complex affine surfaces with branch points dates back to Gunning \cite{gunning78}. In \cite{veech93}, Veech provides a comprehensive treatment of the moduli spaces of complex affine surfaces $\ModA$. There he gives local coordinates on $\ModA$ using the representation variety $\Hom(\pi_1(\Sigma_{g,n},*), \mathrm{Aff}(\bC))$ along with the relative locations of the cone points and also shows that $\TeichA(\bfm)$ is a complex analytic submanifold of
$\Teich \times \Hom(\pi_1(\Sigma_{g,n}, *), \bC^*)$, along with many other results. 
In Sections \ref{S:affine} and \ref{S:Moduli}, we provide an alternate and
self-contained treatment of some of Veech's work on complex affine
surfaces and use these results as a foundation upon which to
understand the topology of the moduli space of dilation
surfaces, which Veech did not specifically study in his paper.

Theorem \ref{T:Components} is reminiscent of the classification of
components of strata of translation surfaces by Kontsevich and Zorich
\cite{kz03}, who showed that each stratum has up to three components:
a hyperelliptic component in the strata $\mathcal{H}(2g-2)$ and
$\mathcal{H}(g-1,g-1)$, two other connected components differentiated
by parity of spin structure when the orders of the zeros are all even,
and one other connected component otherwise. Boissy \cite{boissy}
obtained a similar description of components of strata of meromorphic
differentials. It would be interesting to determine whether Theorem
\ref{T:Components} could be used to deduce these results.  

We also note that various authors have studied dynamical questions about dilation surfaces. For more on this, see \cite{BFG}, \cite{BGT} \cite{BoGh}, \cite{BS}, \cite{DFG}, \cite{ghazouani}, \cite{tahar}, and \cite{Wang}.

\section{Affine and dilation surfaces}
\label{S:affine}

In this section, we define a \emph{complex affine surface} roughly as
the data of a closed Riemann surface $X$ together with a set of marked
points $P = \{p_1, \ldots, p_n\}$, a flat line bundle
$L\to X \setminus P$, and a holomorphic section $\omega$ of
$L\otimes \Omega_X$ over $X\setminus P$.  We define the order of
vanishing $m_i\in \cx$ of such a section at $p_i$ in
\S,\ref{sub:affine}.

Given a surface $X$, which we will always assume to be closed, and a
set of marked points $P\subset X$, we denote by $A(X, P)$
the set of affine structures on $X$ with cone points at $P$, and by
$A(X, P, \bfm)$ the set of affine structures with order
$m_i$ at $p_i$. We define in
\S\,\ref{sub:affinecone} the \emph{exponential action}, as introduced by Veech, of $\Omega(X)$
on $A(X, P ,\bfm)$ and of $\Omega(X, P)$ (the space of
holomorphic one-forms with at worst simple poles along $P$) on
$A(X, P)$.  The action of $\Omega(X, P)$ is analogous and modifies the
cone angles by adding the residues.  In \S\,\ref{sub:affinecone}, we
prove the following theorem of Veech which shows
that given marked points $P$ and cone angles $\bfm$ summing to $2g-2$,
there is an affine structure with these data which is unique up to the
exponential action.

\begin{theorem}[{\cite[Theorem~1.13]{veech93}}]
  \label{thm:veech}
  The set of affine structures $A(X, P)$ is nonempty, and
  $A(X, P, \bfm)$ is nonempty if and only if
  $\sum m_i = 2g-2$.  The exponential actions of $\Omega(X, P)$ and
  $\Omega(X)$ are free and transitive.
 \end{theorem}

In other words, $A(X, P, \bfm)$ and  $A(X, P)$ are
affine spaces modeled on $\Omega(X)$ and $\Omega(X, P)$ respectively.

We moreover define a dilation surface as an affine surface whose flat
bundle $L$ has positive real holonomy.  As a corollary of
Theorem~\ref{thm:veech}, we show that the data of $X$, $P$, and orders
of vanishing define a dilation structure uniquely up to the action of
a lattice $\Lambda\subset \Omega(X)$.

\subsection{Flat bundles}
\label{sub:flat}

We start by recalling some background on flat bundles. Fix a basepoint $*$ of $X$. Given $\chi \in \Hom(\pi_1(X, *), \mathbb{C}^*) = H^1(X, \mathbb{C}^*)$,  define a flat holomorphic line bundle \begin{equation*}
  \flat[\chi, *] = (\widetilde{X} \times \cx) / \pi_1(X, *),
\end{equation*}
where $\pi_1(X, *)$ acts on the universal cover $\widetilde{X}$ of $X$
as the deck transformation group and on the second factor by the
character $\chi$.  We will usually omit the base-point from the
notation, as two flat bundles with the same holonomy character are
isomorphic. This defines a homomorphism $H^1(X, \cx^*)\to \Pic^0(X)$,
where $\Pic^0(X)$ is the group of degree zero holomorphic line bundles
on $X$, which can be seen via the commutative diagram in Figure
\ref{F:CD}.
\begin{figure}[ht]
    \begin{tikzcd}
    &  & 0\ar[d]  & & \\
    & & \Omega(X) \ar[d] & & \\
    0 \ar[r] &  H^1(X, \zed) \ar[r]\ar[d] & H^1(X, \cx) \ar[r]\ar[d] & H^1(X, \cx^*) \ar[r]\ar[d] & 0 \\
    0 \ar[r] & H^1(X, \zed) \ar[r] & H^1(X, \mathcal{O}_X) \ar[r]\ar[d] & \Pic^0(X) \ar[r] & 0 \\
    & & 0 & &.
\end{tikzcd}
\caption{}\label{F:CD}
\end{figure}
The bottom row comes from the exponential sheaf sequence on $X$. The vertical row is part
of the long exact sequence associated to the short exact sequence of sheaves,
\begin{equation*}
    \begin{tikzcd}
    0 \ar[r] &\mathbb{C}_X \ar[r] &\mathcal{O}_X \ar[r, "d"] &\Omega_X \ar[r] &0. \end{tikzcd}
\end{equation*}
Given $\alpha \in \Omega(X)$, define the
$\chi_\alpha \in H^1(X, \cx^*)$ by
$\chi_\alpha(\gamma) = e^{\int_\gamma \alpha}$.  The following classical
result, which appears in Gunning \cite[\S~8]{gunning_lectures},
follows by a straightforward diagram chase.

\begin{prop}
	\label{prop:kernel}
The map $H^1(X, \cx^*) \to \Pic^0(X)$ is a surjection whose kernel is $\{\chi_\alpha\}_{\alpha \in \Omega(X)}$.
\end{prop}
\begin{rem}\label{R:Section}
  Let $\pi\colon \widetilde{X} \rightarrow X$ be the universal cover
  and $\alpha \in \Omega(X)$. Let $p$ be any preimage of $*$ on
  $\widetilde{X}$. The map
  $\tilde{s}_\alpha(x) = e^{\int_p^x \pi^*\alpha}$ is
  $\pi_1(X, *)$-equivariant and hence descends to a nowhere vanishing
  holomorphic section $s_\alpha$ of $L_{\chi_\alpha}$, which shows
  that $L_{\chi_\alpha}$ is holomorphically trivial.
\end{rem}

\subsection{Affine surfaces} 
\label{sub:affine}

We will now define flat bundles over punctured surfaces and then define a complex affine surface as a meromorphic section of such a bundle. 

\begin{definition}
  Suppose that $L$ is a flat bundle over $\Delta^*$, the unit disk in
  $\mathbb{C}$ punctured at $0$, so that the monodromy of a positively
  oriented loop around $0$ is $a \in \mathbb{C}^*$. Let $m$ be a
  complex number so that $e^{2\pi i m} = a$. Any section of $L$ can be
  written as $s(z) = z^m f(z)$ where
  $f\colon \Delta^* \to \mathbb{C}$ is holomorphic. We say that
  $s$ is \emph{meromorphic} if $f$ is, in which case its \emph{order
    at $0$} is $\ord_0(f) + m$ (which does not depend on
  $m$).

  Similarly, we say a section $\omega$ of $L\otimes\Omega_{\Delta^*}$ is
  \emph{meromorphic} if $\omega/dz$ is a meromorphic section of $L$, and
  define its \emph{order} as $\ord_0 \omega = \ord_0 \omega/dz$.
\end{definition}

Since these definitions are coordinate-independent we may speak of
meromorphic sections of a flat bundle over a punctured surface as well
as the order of vanishing at a puncture.

\begin{definition}
\label{D:ComplexAffineSurface}
A \textit{branched complex affine surface} $(X, P, \chi, \omega)$ consists of a
Riemann surface $X$, a tuple of points
$P = (p_1, \ldots, p_n) \subset X$, a character
$\chi \in H^1(X \setminus P,\cx^*)$, and a meromorphic section
$\omega$ of $L_\chi \otimes \Omega_X$, with no poles or zeros on
$X \setminus P$.  We regard two sections of $L_\chi \otimes \Omega_X$
which differ by a constant multiple as defining the same affine surface.

If $\chi \in H^1(X \setminus P,\mathbb{R}_+)$ (resp.\ $\chi \in H^1(X \setminus P, S^1)$), then the complex affine surface is called a \emph{dilation} (resp.\ \emph{flat}) surface.
\end{definition}

We remark that the classical definition of a complex affine surface is
a $(G,X)$-structure on $X$ where $G$ is the group of complex affine
maps $\{z \mapsto az + b \colon z \in \bC^*, b \in \bC\}$ (see
\cite{gunning78}). The above definition gives a $(G,X)$-structure on
$X \setminus P$ but allows for cone points at $P$.  As we are
primarily interested in closed surfaces of genus greater than $1$, we
will continue to abuse terminology and refer to branched complex
affine surfaces simply as affine surfaces.

The points in $P$ will be called \emph{cone points}. If $\omega$
denotes the meromorphic section of $L_\chi \otimes \Omega_X$, the
\emph{complex cone angle at $p_j$} is
$2\pi\left(1 + \mathrm{ord}_{p_j}(\omega) \right)$. The real part of
the cone angle dictates the angle around the cone point, and the exponentiation of the product of $-2\pi$ and the imaginary part dictates the dilation of the metric around the cone point 
(see Figure \ref{F:dilation-tori}). In particular, the orders $\ord_{p_j} (\omega)$ of cone points of a dilation surface
must have integral real part and those of a flat surface must be
purely real.  This definition allows for $\omega$ to have zeros or
poles at the cone points. 

\begin{figure}[htbp]
\centering
\begin{subfigure}{.45\textwidth}
  \centering
  \includegraphics[width=.95\linewidth]{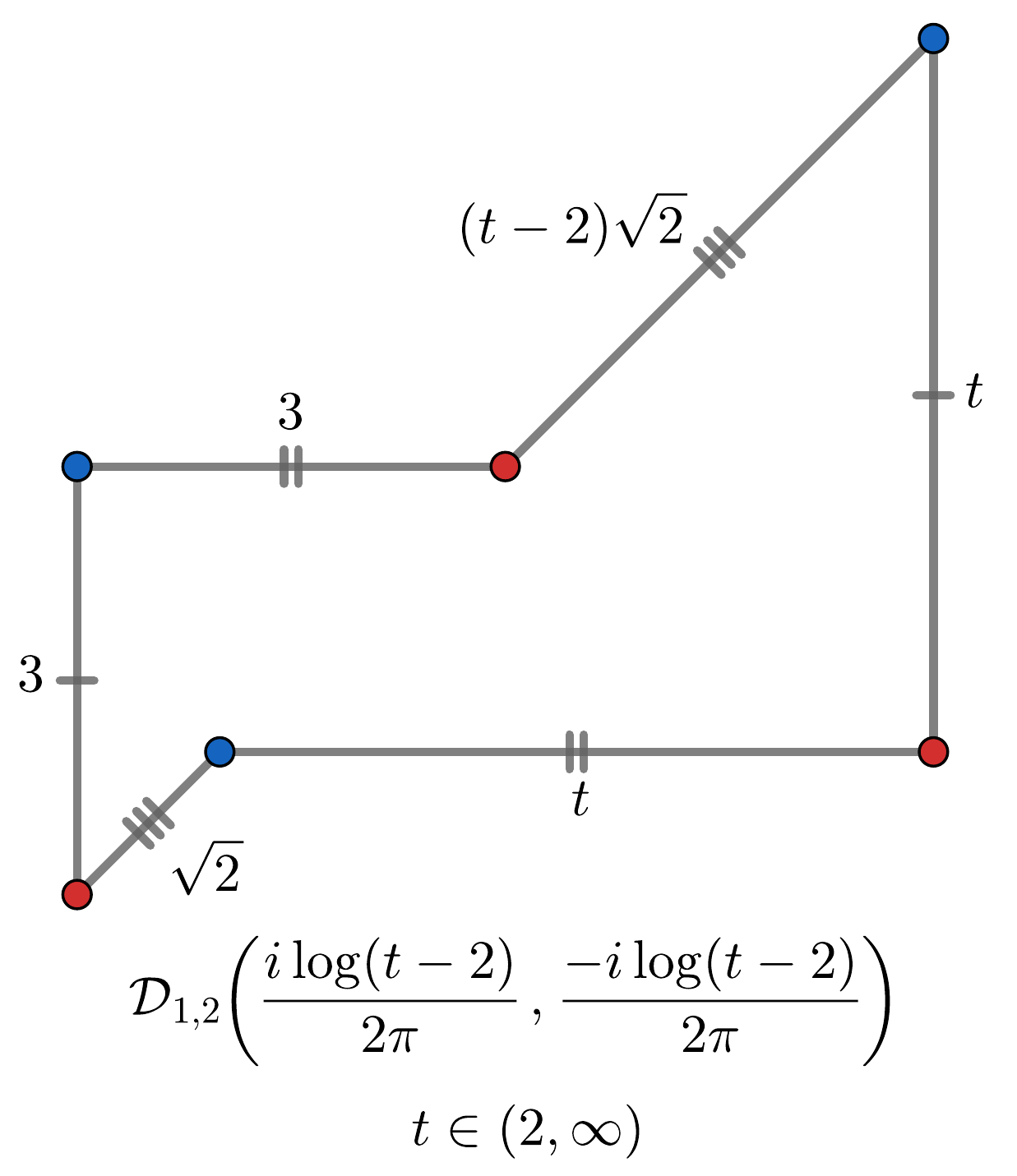}
  \caption{A one parameter family of dilation surfaces no two of which belong to the same stratum. The labels indicate lengths.} 
  \label{F:dilation-tori:sub1}
\end{subfigure}
\hspace{.5cm}
\begin{subfigure}{.45\textwidth}
  \centering
  \includegraphics[width=.95\linewidth]{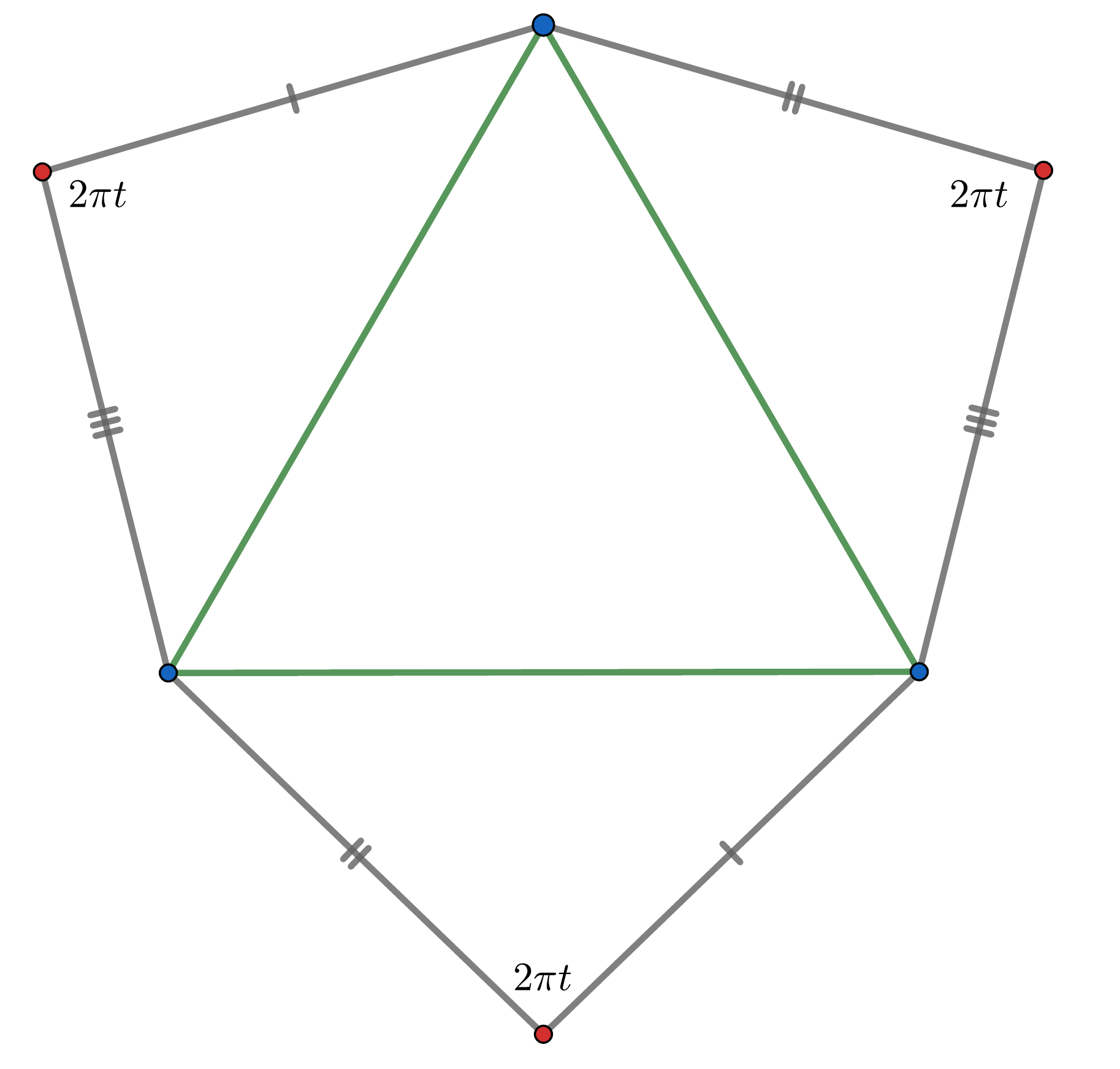}
  \caption{A one parameter family of flat surfaces formed by gluing congruent isosceles triangles to an equilateral triangle. For $t \in (0, 1]$ these surfaces belong to $\mathcal{A}_{1,2}(3t-1, 1-3t)$. }
  \label{F:dilation-tori:sub2}
\end{subfigure}
\caption{Two one-parameter families, one of dilation and the other of flat surfaces, no two of which belong to the same stratum.}
\label{F:dilation-tori}
\end{figure}

\subsection{Flat singular bundles}

We now wish to give an alternative definition of a complex affine
surface in terms of bundles over $X$ rather than the punctured surface
$X \setminus P$.  This will be the most useful point of view for
constructing the moduli space of affine surfaces in
\S\,\ref{S:Moduli}.  We first recall the notion of a flat bundle with
isolated singularities, extending $L_\chi$ over the punctures.  This
extension is not unique and essentially involves a choice of a
logarithm of the holonomy around each puncture.  

Given a finite-type surface $X\setminus P$, we will say that the
\emph{holonomy data} refers to a character
$\chi\in H^1(X\setminus P, \cx^*)$ together with the choice of a
collection $\bfm = (m_1, \ldots, m_n)$ of complex numbers so that
$e^{2 \pi i m_j} = \chi(\gamma_j)$ where $\gamma_j$ is a positively
oriented loop around the puncture $p_j$. Note that since
$\sum \gamma_j=0$ in $H_1(X\setminus P)$, we have $\sum m_j \in \mathbb{Z}$.

An affine surface $(X, P, \chi, \omega)$ as in Definition
\ref{D:ComplexAffineSurface} determines such holonomy data, where for
each puncture $m_i = \ord_{p_i} \omega$.

Following Deligne-Mostow \cite[\S~2]{delignemostow}, we note that
holonomy data defines an extension $L_{\chi, \bfm}$ of $L_\chi$ over
the punctures of $X \setminus P$.  

\begin{definition}
  If $\iota\colon X \setminus P \rightarrow X$ is the inclusion map,
  then define $L_{\chi, \bfm}$ to be the line bundle corresponding to
  the locally free subsheaf of $\iota_* L_\chi$ whose sections are
  meromorphic sections $\sigma$ of $L_{\chi}$ so that
  $\mathrm{ord}_{p_j}(\sigma)\geq m_j$ for all $j$.
\end{definition}

Informally, this definition amounts to declaring that $z^{m_i}$ is a
nonzero local holomorphic section of $L_{\chi, \bfm}$ in a neighborhood of $p_i$.

\begin{definition}\label{D:HolonomyDataModuli}
  Let $V_{X, P} \subseteq H^1(X \setminus P, \cx^*)\times \cx^n$
  denote the collection of pairs $(\chi, \bfm)$ so that
  $e^{2\pi i m_j} = \chi(\gamma_j)$ for all $j$.  Similarly, let
  $V_{X, P}(\bfm) = V_{X, P} \cap (H^1(X \setminus P, \cx^*) \times
  \{\bfm\})$.  The holonomy data of
  a complex affine surface is an element of $V_{X,P}$.
\end{definition}

\begin{prop}[{\cite[Proposition~2.11.1]{delignemostow}}]\label{P:DM}
  $\mathrm{deg}\left( L_{\chi, \bfm} \right) = -\sum m_i$.
\end{prop}
\begin{proof}
  Each element of $V_{X, P}$ defines a line bundle $L_{\chi,
    \bfm}$. Moreover, $(\chi, \bfm)$ and $(\chi', \bfm')$ are in the
  same component if and only if $\sum m_j = \sum m_j'$.

  For any family of line bundles, the degree is a continuous
  function. Since each component of $V_{X, P}$ contains a pair with
  $\chi=1$ and $\bfm$ integral, whose corresponding line bundle is the
  classical bundle $\mathcal{O}(-\sum m_i p_i)$, whose degree is
  $-\sum m_i$, the result follows.
\end{proof}

We are now prepared to give a second equivalent definition of a
complex affine surface.

\begin{definition}\label{D:CAS-scaling}
A \emph{complex affine surface} is a tuple $(X, P, \chi, \bfm)$ consisting of a Riemann surface $X$, a finite set of points $P \subset X$, and holonomy data $(\chi, \bfm)$ so that $L_{\chi, \bfm} \otimes \Omega_X$ is holomorphically trivial.
\end{definition}

\begin{prop}\label{L:Equivalent}
The two definitions $(X,P,\chi,\omega)$ from Definition \ref{D:ComplexAffineSurface} and $(X,P,\chi,\bfm)$ from Definition \ref{D:CAS-scaling} are equivalent. 
\end{prop}

\begin{proof}
If $(X, P, \chi, \omega)$ is a complex affine surface from Definition \ref{D:ComplexAffineSurface}, then let $m_i$ be the order of $\omega$ at $p_i$. By definition, $\omega$ defines a nowhere-vanishing holomorphic section of $L_{\chi, \bfm}$, which is consequently a trivial bundle.

Conversely, if $(X, P, \chi, \bfm)$ is a complex affine surface from Definition \ref{D:CAS-scaling}, then $L_{\chi, \bfm} \otimes \Omega_X$ is holomorphically trivial and hence admits a nowhere-vanishing section $\omega$ that is unique up to scaling.
\end{proof}

\subsection{Affine structures with fixed cone points}
\label{sub:affinecone}

Given a Riemann surface $X$, marked points $P = \{p_1, \ldots, p_n\}$,
and $\bfm = (m_1, \ldots,m_n)$ summing to $2g-2$, we will now show that there
exists a unique complex affine structure up to the exponential action
by $\alpha \in \Omega(X)$, a unique dilation structure up to the
exponential action by $\alpha \in \Omega(X)$ with
$\Im(\alpha) \in H^1(X,2\pi\mathbb{Z})$, and a unique flat structure.

\begin{definition}
\label{def:exponentiation}
Suppose that $(X, P, \chi, \omega)$ is a complex affine surface and
that $\alpha \in \Omega(X)$. Then $s_\alpha \otimes \omega$ is a
meromorphic section of $L_{\chi_\alpha} \otimes L_{\chi}$ with no
poles or zeros on $X \setminus P$ (see Remark \ref{R:Section} for the
definition of $s_\alpha$ and $\chi_\alpha$) and so
$(X, P, \chi_\alpha \chi, s_\alpha \otimes \omega)$ is a complex
affine surface, which we denote by $\alpha \cdot \omega$.  We call
this the \emph{exponential action}.
\end{definition}

\begin{remark}
  We note that when $\alpha \in \Omega(X)$, the exponential action
  does not change the location or cone angles of a complex affine
  surface. If we generalize this action to allow for $\alpha$ 
  to have simple poles, then at every pole $p_i$ of $\alpha$, the section
  $s_\alpha \otimes \omega$ has order
  $\ord_{p_i} \omega + \Res_{p_i}\alpha$ at $p_i$. That is, the
  complex cone angle at $p_i$ changes by $2\pi(\Res_{p_i}\alpha)$
  under the action of $\alpha$.
\end{remark}

We end this section by proving results about the existence of complex
affine, dilation, and flat structures with prescribed cone angle
data. Theorem \ref{thm:veech} is due to Veech and Corollary~\ref{C:Troyanov} is due to Troyanov when $\Re(\ord_{p_i} \omega) > -1$
at every cone point $p_i$. 

\begin{proof}[Proof of Theorem~\ref{thm:veech}]
  The degree of $L_{\chi, \bfm} \otimes \Omega_X$ is $2g-2 - \sum m_i$
  by Proposition~\ref{P:DM}.  Since a trivial bundle has degree $0$,
  an affine structure must have $\sum m_i = 2g-2$.
  
  Let $\chi \in H^1(X \setminus P, \cx^*)$ be any character so that
  $\chi(\gamma_j) = e^{2 \pi i m_j}$ for all $j$, where $\gamma_j$ is
  a positively oriented loop around $p_j$. By
  Proposition~\ref{prop:kernel}, there is $\rho \in H^1(X, \cx^*)$ so
  that $L_\rho^{-1} = L_{\chi, \bfm} \otimes \Omega_X$. Therefore,
  $L_{\rho \chi, \bfm} \otimes \Omega_X$ is holomorphically trivial
  and hence $(X, P, \rho \chi, \bfm)$ is a complex affine surface. We
  are done by Proposition~\ref{L:Equivalent}.

  Suppose now that $(X, P, \chi_1, \bfm)$ and $(X, P, \chi_2, \bfm)$
  are complex affine structures that define meromorphic sections
  $\omega_i$ of $L_{\chi_i} \otimes \Omega_X$ that do not have zeros
  or poles on $X \setminus P$. Then $\frac{\omega_1}{\omega_2}$ is a
  nowhere vanishing section of $L_{\chi_1 \chi_2^{-1}}$, which we view
  as a bundle over $X$ (not merely $X \setminus P$). By Proposition
  \ref{prop:kernel}, $\chi_1 \chi_2^{-1} = \chi_\alpha$ for some
  $\alpha \in \Omega(X)$. It follows that
  $\omega_1 = s_\alpha \otimes \omega_2$, as desired.

  The analogous statements for the action of $\Omega(X, P)$ are proved similarly.
   \end{proof}
  
 \begin{cor}\label{C:DilationFiber}
   Given a closed Riemann surface $X$ of genus $g$, points
   $P = (p_1, \ldots, p_n)\subset X$, and complex numbers
   $\bfm = (m_1, \ldots, m_n)$ with $\sum m_i = 2g-2$ and
   $\Re(m_i) \in \mathbb{Z}$, there is a dilation structure on $X$
   with each $p_i$ a cone point of complex cone angle $2\pi(m_i+1)$.
   Any two dilation structures differ by the exponential action of
   $\alpha \in \Omega(X)$ where
   $[\mathrm{Im}(\alpha)] \in H^1(X, 2\pi\mathbb{Z})$.
 \end{cor}

 \begin{rem}
   In other terms, the set of dilation structures $D(X, P, \bfm)$ on
   $X$ with cone points at $P$ of fixed orders $\bfm$ is a torsor for
   $H^1(X, \zed)$ via the exponential action, where we identify
   $\Omega(X)$ with $H^1(X, \mathbb{R})$ by $\alpha \mapsto [\mathrm{Im}(\alpha)/2\pi]$.
 \end{rem}

 \begin{proof} By Theorem \ref{thm:veech}, let $(X, P, \chi, \bfm)$ be
   a complex affine structure. Notice that
   $\frac{|\chi|}{\chi} \in H^1(X, S^1)$ (the fact that $\chi(\gamma)$
   is a positive real number for any loop $\gamma$ around a puncture
   is what allows us to take the cohomology of $X$ and not just
   $X \setminus P$). There is a lattice of elements
   $\rho \in H^1(X, \mathbb{R})$ so that
   $e^{i\rho} = \frac{|\chi|}{\chi}$. By the Hodge theorem, any such
   cohomology class is the real part of a unique holomorphic
   $1$-form. Choose $\alpha \in \Omega(X)$ so that
   $[\Im \alpha] = \rho$. Therefore,
   $(X, P, \chi_{\alpha} \chi, \bfm)$ is a complex affine structure
   and, since $\chi_{\alpha} \chi = e^{[\Re\alpha]}|\chi| \in H^1(X,\mathbb{R})$, it is a dilation
   surface structure. The final claim is immediate from the final
   claim of Theorem \ref{thm:veech}.
\end{proof}

 \begin{cor}[Troyanov, \cite{troyanov}]\label{C:Troyanov}
   Given a closed Riemann surface $X$ of genus $g$, points
   $P = \{p_1, \ldots, p_n\}\subset X$, and real
   $\bfm = (m_1, \ldots, m_n)$ with $\sum m_i = 2g-2$, there is a
   unique flat structure (up to scaling) on $X$ with each $p_i$ a cone
   point with cone angle $2 \pi (m_i+1)$.
 \end{cor}
 
 As a warning, we note that in Figure \ref{F:dilation-tori:sub2} we produced a one-parameter family of flat structures in $\mathcal{A}_{1,2}$ that map to a single point in $\mathcal{M}_{1,2}$. Therefore, it is crucial to bear in mind that, for Corollary \ref{C:Troyanov}, given $(X, P) \in \cM_{g,n}$ there is a unique flat structure on $(X, P)$ in $\mathcal{A}_{g,n}(\bfm)$, but not necessarily in $\mathcal{A}_{g,n}$.

\begin{proof}
  By Theorem \ref{thm:veech}, there is an affine structure
  $(X, P, \chi, \bfm)$.  Notice that
  $|\chi|^{-1} \in H^1(X, \mathbb{R}_+)$ (the fact that
  $\chi(\gamma) \in S^1$ for any loop $\gamma$ around a puncture is
  what allows us to take the cohomology of $X$ and not just
  $X \setminus P$). There is a unique $\rho \in H^1(X, \mathbb{R})$ so
  that $e^{\rho} = |\chi|^{-1}$. By the Hodge theorem, $\rho$ is the
  real part of a unique holomorphic $1$-form $\alpha$. Then,
  $e^{\Re \alpha} = |\chi|^{-1}$. Therefore,
  $(X, P, \chi_\alpha \chi, \bfm)$ is a complex affine structure and,
  since $\chi_\alpha \chi = e^{i[\Im\alpha]}\frac{\chi}{|\chi|} \in H^1(X,S^1)$, it is a flat
  surface. Uniqueness follows from the final claim of
  Theorem \ref{thm:veech} as any other $(X,P,\chi_\beta\chi,\bfm)$ would not have holonomy in $H^1(X,S^1)$. 

\end{proof}

\section{Moduli spaces of affine and dilation surfaces}\label{S:Moduli}

We now construct the moduli spaces of affine surfaces and dilation surfaces.  As usual, we will first construct their corresponding
\Teichmuller spaces and take the quotient by the mapping class group.

Let $\Sigma$ be a fixed genus $g$ surface with $n+1$ marked points
$\{s_1, \ldots, s_{n+1}\}$.  We denote by $S$ the first
$n$ of these points.  Let $\Teich$ be the
\Teichmuller space of $(\Sigma, S)$.  The product $B = \TeichV$ is
the moduli space of marked surfaces together with holonomy data (the definition of $V_{\Sigma, S}$ appears in Definition \ref{D:HolonomyDataModuli}).  Let
$\mathcal{X}\to B$ denote its universal curve, which is the product of the universal curve of $\Teich$ with $V_{\Sigma, S}$.
We denote by $P_i$ the section of $\mathcal{X}$ corresponding to the
$i$th marked point and denote by $P$ their union.

Recall that the universal curve over $\Teich$ may be defined as the quotient
\begin{equation*}
  \Teich[g, n+1] / \pi_1(\Sigma \setminus \{s_1, \ldots s_n\}, s_{n+1}),
\end{equation*}
where the group acts via the point-pushing map
$$\pi_1(\Sigma \setminus \{s_1, \ldots s_n\}) \to \Mod[g, n+1].$$ Define
the \emph{universal flat bundle}
$\mathcal{L} \to \mathcal{X} \setminus P$ as the quotient
\begin{equation*}
  \mathcal{L} = (\Teich[g, n+1] \times V_{\Sigma, S} \times \cx )/ \pi_1(\Sigma \setminus \{s_1, \ldots s_n\}, s_{n+1}),
\end{equation*}
where the group acts on $\cx$ via the character determined by the
holonomy data. Let $\mu_i\colon \cL \ra \bC$ and and
$\bfmu\colon\cL \ra \bC^n$ be the functions that record the $i$th entry
of $\bfm$ or all of $\bfm$ respectively.

The product $\mathcal{L} \otimes \Omega_{\mathcal{X}/B}$, where
$\Omega_{\mathcal{X}/B}$ is the relative cotangent bundle, is a line
bundle over $\mathcal{X} \setminus P$, which we may extend to a line
bundle $\mathcal{K}\to\mathcal{X}$ as we did in the previous section.  More
precisely, on an open set $U\subset\mathcal{X}$ meeting $P_i$, let
$z_i$ be a local holomorphic function with $(z_i) = P_i \cap U$.  We
extend $\mathcal{K}$ over $U$ by declaring $z_i^{\mu_i} dz_i$ to be a
nonzero holomorphic section over $U$.

As the restriction of $\mathcal{K}$ to any fiber of the universal
curve is degree $0$, it defines a holomorphic section
$\sigma\colon B\to \Jac(\mathcal{X}/B)$ of the relative Jacobian
variety (which can be defined as the quotient of the dual of the Hodge
bundle over $B$ by $H_1(\Sigma, \zed)$).  The preimage $\sigma^{-1}(0)$
of the zero-section is the locus of surfaces $X$ with holonomy data
$(\chi, \bfm)$ such that $L_{\chi, \bfm}$ is trivial.  We define the
\emph{\Teichmuller space of affine surfaces}
$\TeichA = \sigma^{-1}(0) \subset B$.  We have a holomorphic map
$\bfmu = (\mu_1, \ldots, \mu_n) \colon \TeichA \to \cx^n$ which
records the orders of the affine structure at each $p_i$.  Given
complex number $\bfm = (m_1, \ldots, m_n)$ whose sum is $2g-2$ we
denote by $\TeichA(\bfm)\subset \TeichA$ the fiber of $\bfmu$ over
$\bfm$. The mapping class group acts on these spaces, and we define
the moduli space of affine surfaces $\ModA(\bfm)$ and $\ModA$ to be
their quotients.

\begin{theorem}
  $\TeichA(\bfm)$ and $\TeichA$ are complex submanifolds of $B$.
  Moreover, the forgetful maps $\TeichA(\bfm)\to\Teich$ and
  $\TeichA\to\Teich$ are submersions.
\end{theorem}

\begin{proof}
  Let $\cX'$ denote the universal curve over $\Teich$. The projection
  $\pi\colon B \ra \Teich$ induces a map
  $\pi_{\Jac}\colon \Jac(\cX/B) \ra \Jac(\cX'/\Teich)$.  Since
  $\TeichA$ is also the preimage $(\pi_{\Jac} \circ \sigma)^{-1}(0)$,
  to show that $\TeichA$ is a complex submanifold, it suffices to show
  that $\pi_{\Jac} \circ \sigma$ is a submersion.  For this, it
  suffices to show that for every $\bfm$ and $X \in \Teich$, the
  restriction of $\sigma$ to
  $\sigma_X\colon\{X\} \times V_{(\Sigma, S)}(\bfm) \ra \Jac(X)$ is a
  submersion.

  Let $(X, P,\chi_0, \bfm)$ be a complex affine structure on $X$,
  which exists by Theorem \ref{thm:veech}.  Any element of
  $V_{(\Sigma, S)}(\bfm)$ can be written as $\rho \chi_0$ where
  $\rho \in H^1(X, \mathbb{C}^*)$. Moreover,
  $\sigma_X(\rho \chi_0) = L_{\rho} \otimes L_{\chi_0}\otimes
  \Omega_X$.  Therefore, $\sigma_X$ is a translate of the natural map
  $H^1(X, \cx^*)\to \Pic^0(X)$, which is the rightmost vertical arrow
  in Figure \ref{F:CD}. As seen in Figure \ref{F:CD}, the derivative
  of this map is $H^1(X, \cx) \to H^1(X, \mathcal{O}_X)$, which is
  surjective by the vertical short exact sequence.  This completes the
  proof that $\TeichA$ is a complex submanifold.
  
  To see that $\TeichA(\bfm)$ is a complex submanifold of $\TeichA$,
  we note that the map $\bfmu\colon\TeichA\to\cx^n$, recording the
  cone angles is a submersion.  This follows from the observation that
  the restriction of $\bfmu$ to an $\Omega(X, P)$-orbit is affine and
  surjective.

Given a point $p \in \TeichA(\bfm)$, 
\[ T_p \TeichA(\bfm) = \left( d\left( \pi_{\Jac} \circ \sigma \right)_p \right)^{-1} \left( T_{\pi_{\Jac} \circ \sigma(p)} 0 \right), \]
where $0$ denotes the zero section of $\Jac(\cX'/\Teich)$. In particular, the map $(\pi_{\Jac} \circ \sigma) \big|_{\TeichA(\bfm)}$ from $\TeichA(\bfm)$ to $0$ is a submersion. Since the forgetful map from $\TeichA(\bfm)$ to $\Teich$ is a composition of $(\pi_{Jac} \circ \sigma) \big|_{\TeichA(\bfm)}$ and the projection from $0$ to $\Teich$, both of which are submersions, the final claim follows.
\color{black}
\end{proof}

\begin{proof}[Proof of Theorem \ref{thm:affine_covering}]
  Given a point $p \in \Teich$, there is an open neighborhood $U$ of
  $p$ and a holomorphic section $s\colon U \ra \TeichA(\bfm)$ (by the
  implicit function theorem).  Let $\Omega U$ denote the Hodge bundle
  over $U$ and choose a trivialization that identifies it with
  $U \times \mathbb{C}^g$. The map
  $f\colon U \times \mathbb{C}^g \ra \pi^{-1}(U)$ that sends a point
  $(u, \alpha)$ to the exponential action of $\alpha$ on $s(u)$ is a
  biholomorphism by Theorem \ref{thm:veech}. This construction gives
  $\pi$ the structure of an affine bundle. In particular,
  $\pi\colon \TeichA(\bfm) \ra \Teich$ is an $\Omega \Teich$ torsor.
  Such torsors are classified by $H^1(\Teich, \Omega \Teich)$.  Since
  $\Teich$ is a Stein manifold, any higher cohomology group is
  trivial, so in fact $\TeichA(\bfm)$ is biholomorphically equivalent
  to $\Teich\times \cx^{g}$.

  As the forgetful maps are equivariant with respect to the action of
  the mapping class group, and moreover this action preserves the
  affine structures of the fibers, it follows that
  $\ModA(\bfm)\to \moduli$ and $\ModA\to \moduli$ are holomorphic
  affine bundles.

  To identify the fiber of $\ModA(\bfm)$ over $(X, P)$ with
  $\Omega(X)/\mathrm{Aut}(X,P)$, we note that the unique flat metric
  $\rho$ from Corollary~\ref{C:Troyanov}
   with cone angles $\bfm$ at $P$ is an affine structure which
  must be fixed by the action of $\mathrm{Aut}(X,P)$.  The
  exponential action then gives via Theorem~\ref{thm:veech} an $\mathrm{Aut}(X, P)$-equivariant bijection
  $\omega \mapsto \omega\cdot\rho$ between $\Omega(X)$ and the set of
  affine structures $A(X, P, \bfm)$ with these cone angles.
  The fibers of $\ModA$ are identified similarly.
\end{proof}

We note that by the remark after Definition \ref{def:exponentiation}
of the exponential action, similar reasoning as in the proof of
Theorem \ref{thm:affine_covering} implies that $\TeichA$ is a
holomorphically trivial bundle of affine surfaces over $\Teich$ whose
fiber over $(X,P)$ is an affine space modeled on the bundle of
meromorphic one-forms with at worst simple poles at the marked points.

\begin{remark}
It seems to be an interesting open question whether or not the map $\ModA\to \moduli$ is the projection map of an orbifold vector bundle.
\end{remark}

We similarly define the \Teichmuller spaces and moduli spaces of
dilation surfaces  $\TeichD(\bfm)\subset\TeichA(\bfm)$ and
$\ModD(\bfm)\subset\ModA(\bfm)$, where $\bfm \in \cx^n$ is a tuple of
complex numbers with integral real part which sum to $2g-2$, to be the locus of affine surfaces
with real holonomy and orders of vanishing specified by $\bfm$.

\begin{proof}[Proof of Theorem~\ref{thm:dilation_covering}]
  Choose a section $\rho\colon\Teich\to\TeichA(\bfm)$, let the holonomy be
  $\chi\colon\Teich\to H^1(\Sigma\setminus S, \cx^*)$, and choose a lift $\tilde{\chi}\colon \Teich \to
  H^1(\Sigma\setminus S, \cx)$ so that $e^{\tilde{\chi}} = \chi$.
  Given a cohomology class $\gamma\in H^1(\Sigma , 2\pi i \bZ)$, let
  $\alpha_\gamma$ be the real analytic section of the Hodge bundle
  over $\Teich$ whose cohomology class has imaginary part specified by
  \begin{equation*}
    [\Im \alpha_\gamma] = \gamma - \Im \tilde{\chi}.
  \end{equation*}
  The exponential action then gives a real analytic isomorphism
  $F\colon \Teich \times  H^1(\Sigma , 2\pi i \bZ) \to \TeichD(\bfm)$,
  defined by
  \begin{equation*}
    F(X, \gamma) = e^{\int \alpha_\gamma} \rho(X).
  \end{equation*}
  In particular, $\TeichD\to \Teich$ is a covering map, as is the
  quotient $\ModD\to \moduli$.
\end{proof}

\begin{rem}
  As Figure \ref{F:dilation-tori} illustrates, it is often natural to
  gather genus $g$ dilation surfaces into a \emph{superstratum}
  $\mathcal{H}^{dil}(\mathbf{\kappa})$ where $\kappa$ consists of
  integers that sum to $2g-2$ and $\mathcal{H}^{dil}(\mathbf{\kappa})$
  is the set of dilation surfaces whose cone points orders have real
  parts given by $\kappa$. The proof of Theorem
  \ref{thm:dilation_covering} makes it clear that the lift of the
  superstratum to Teichm\"uller space is real analytic submanifold
  foliated by strata of dilation surfaces and whose fiber over a
  surface $X$ corresponds to the holomorphic $1$-forms in
  $\Omega(X, P)$ for which the imaginary part of the period of any
  essential simple closed curve belongs to $2\pi \mathbb{Z}$.
\end{rem}

\section{Framings and the topology
  of $\ModD$}

Suppose that $(X, P, \chi, \omega)$ is a dilation surface. Recall that
the section $\omega$ is only defined up to multiplication by a
constant in $\mathbb{C}^*$.  Fix a specific section $\omega$. Since the character $\chi$ assumes positive real values, $(\widetilde{X} \times \mathbb{R})/\pi_1(X, *)$ and $(\widetilde{X} \times i\mathbb{R})/\pi_1(X, *)$ are oriented real line bundles that are subbundles of $L_\chi$. The \emph{horizontal vector field} of $X \setminus P$ is given by elements $v$ of the unit tangent bundle so that $\omega(v)$ belongs to $(\widetilde{X} \times \mathbb{R}_{> 0})/\pi_1(X, *)$.  This vector field defines a \emph{framing}, i.e. a trivialization $f\colon  T^1(X \setminus P) \rightarrow (X \setminus P) \times S^1$. If we had chosen a section $(re^{i\theta}) \omega$, where $r$ and $\theta$ are positive real numbers, instead of $\omega$, then the trivialization would change by postcomposing by multiplication by $e^{-i\theta}$ on the $S^1$ factor.

Given a $C^1$ immersion $\gamma\colon S^1 \ra X \setminus P$, its
derivative is a map to the unit tangent bundle, which can be
identified via the framing with $(X \setminus P) \times S^1$. The
composition of the derivative and the projection onto the $S^1$ factor
is a map between circles, whose degree $\tau_\omega(\gamma)$
 is called the \emph{turning number of $\gamma$}. Since replacing $\omega$ with $c\omega$ for $c \in \mathbb{C}^*$ only changes the framing by postcomposition with a rotation, the turning number of $\gamma$ is unaffected. Therefore, the turning number function is well-defined for a dilation surface.

\begin{theorem}\label{T:TurningNumber}
Let $(X, P, \chi_1, \omega_1)$ and $(X, P, \chi_2, \omega_2)$ be
dilation surfaces. Let $\alpha \in \Omega(X)$ so that $\omega_2 =
s_\alpha \otimes \omega_1$. Then for any $C^1$ embedded loop
$\gamma$, \[ \tau_{\omega_2}(\gamma) - \tau_{\omega_1}(\gamma) =
  \frac{1}{2\pi}\int_\gamma \mathrm{Im}(\alpha). \]
\end{theorem}
\begin{proof}
Let $Y$ be the simply connected Riemann surface and $\pi\colon Y \ra X \setminus P$ the uniformization map.  There are holomorphic functions $f_i(z)$ so that $\pi^*\omega_i = f_i(z)dz$ for $i \in \{1, 2\}$. Fixing $z_0 \in Y$, and letting $A(z) := \int_{z_0}^z \pi^* \alpha$, we have that 
\[ f_2(z) = e^{A(z)} f_1(z). \]
Let $s\colon [0,1] \ra Y$ be a $C^1$ arc whose projection to $X$ is a
smoothly embedded loop $\gamma$. Then $\tau_{\omega_i}(\gamma)$ is the
degree of $\frac{f_i(s)s'}{|f_i(s)s'|}$, which may be taken to be a
map between circles. Given two maps between circles $g_1$ and $g_2$,
$\mathrm{deg}(g_1 \cdot g_2) = \mathrm{deg}(g_1) +
\mathrm{deg}(g_2)$. Since the degree of the map
$\frac{e^{A(s)}}{|e^{A(s)}|} = e^{ i \int_{z_0}^s
  \pi^*\mathrm{Im}(\alpha)}$ is $\frac{1}{2\pi}\int_\gamma
\mathrm{Im}(\alpha)$,  we are done.
\end{proof}

\begin{cor}\label{C:turning}
Two dilation surface structures on a fixed Riemann surface (with cone points marked) coincide if and only if they determine the same framing.
\end{cor}
\begin{proof}
By Corollary \ref{C:DilationFiber}, we are in the setting of Theorem \ref{T:TurningNumber}, whose notation we adopt. If the two framings coincide, then the two dilation surfaces have identical turning numbers. Thus, $\mathrm{Im}(\alpha) = 0$ and hence $\alpha = 0$, which implies that the two dilation surfaces are equal.
\end{proof}

The turning number of an immersed curve may be generalized to define
the turning number of any loop in $T^1(X \setminus P)$.  This defines
a class in $H^1(T^1(X \setminus P), \zed)$ (see \cite{johnson}), which
is part of the exact sequence

\begin{equation}
  0 \to H^1(X\setminus P, \zed)\to H^1(T^1(X \setminus P), \zed) \to
  H^1(S^1, \zed)\to 0  \tag{4.1}\label{eq:exact sequence},
\end{equation}
and this construction identifies $F(X, P, \boldsymbol{r})$ with the
affine subset of classes whose pairing with an oriented tangent circle
is $1$ and whose pairing with the tangent vector field of $\gamma_i$
is $r_i$.  It follows from this exact sequence that
$F(X, P, \boldsymbol{r})\subset H^1(T^1(X \setminus P), \zed)$ is a
$H^1(X, \zed)$-torsor.

\begin{proof}[Proof of Theorem~\ref{thm:dilation_covering2}]
  It follows from Corollary~\ref{C:DilationFiber} that $D(X,
  P, \bfm)$, the set of dilation structures on $X$ with cone points at $P$ of
  orders $\bfm$, is a $H^1(X, \zed)$-torsor, as is the set $F(X, P,
  \bfr)$ of framings by \eqref{eq:exact sequence}.  The forgetful map
  $D(X, P, \bfm)\to F(X, P, \bfr)$ is $H^1(X, \zed)$-equivariant by
  Theorem~\ref{T:TurningNumber}, so it must be a bijection.  The
  forgetful map $\ModD(\bfm) \to \ModF(\bfr)$ is then a bijection on
  each fiber, so must be an isomorphism of covering spaces.
\end{proof}

\begin{proof}[Proof of Theorem~\ref{T:Components}]
  Since $\ModD(\bfm)\to \moduli$ is a covering space equivalent to
  $\ModF(\bfr)$, and the fundamental group of the base is the mapping
  class group $\Mod$, the components of $\ModD(\bfm)$ are in bijection
  with the orbits of the $\Mod$-action on the set $F(X, P, \bfr)$ of
  framings.  These were classified by Kawazumi in Theorem \ref{T:Kawazumi} below.
 \end{proof}

  \begin{theorem}[Kawazumi \cite{kawazumi}]  \label{T:Kawazumi}
Given two framings $f_1$ and $f_2$ on $X - P$ with $\rho(f_1) = \rho(f_2)$, $f_1$ and $f_2$ are in the same $\mathrm{Mod}_{g,n}$-orbit if and only if one of the following occurs:
\begin{enumerate}
    \item $g=0$,
    \item $g=1$ and $\nu(f_1) = \nu(f_2)$,
    \item $g \geq 2$ and $\rho(f_1)$ contains an even number.
    \item $g \geq 2$, $\rho(f_1)$ consists entirely of odd numbers, and $\mathrm{Arf}(f_1) = \mathrm{Arf}(f_2)$.
\end{enumerate}
\end{theorem}

\bibliography{mybib}{}
\bibliographystyle{halpha}
\end{document}